\documentclass[final]{dmtcs-episciences}

\usepackage[utf8]{inputenc}
\usepackage{subfigure}

\usepackage{amsmath,amssymb,amsthm,amscd,verbatim}
\usepackage[noabbrev,capitalize]{cleveref}
\usepackage{mathtools}
\usepackage{hyperref}
\usepackage[shortlabels]{enumitem}
\newlist{tightenum}{enumerate}{3}
\setlist[tightenum,1]{noitemsep,nosep,
                        label=\rm(\alph*),
                        ref  =\alph*}
\setlist[tightenum,2]{noitemsep,nosep,
                         label=\rm(\roman*),
                         ref  =\roman*}
\setlist[tightenum,3]{noitemsep,nosep,
                         label=\rm(\roman*),
                         ref  =\roman*}

\newtheorem{thm}{Theorem}
\newtheorem{lem}[thm]{Lemma}

\newtheorem*{thm*}{Theorem}
\newtheorem{qn}[thm]{Question}

\theoremstyle{definition}

\theoremstyle{remark}

\newcommand{\dep}{depth}

\newcommand{\poly}{\ensuremath{\mathrm{poly}}}

\newcommand{\scol}{\ensuremath{\mathrm{scol}}}

\newcommand{\sreach}{\ensuremath{\mathrm{SReach}}}

\newcommand{\tangle}{\ensuremath{\mathrm{tn}}}
\newcommand{\bn}{\ensuremath{\mathrm{bn}}}

\newcommand{\GR}{\ensuremath{\mathrm{grid}}}

\newcommand{\link}{\ensuremath{\mathrm{link}}}
\newcommand{\WL}{\ensuremath{\mathrm{well}}}

\renewcommand{\le}{\leqslant}
\renewcommand{\leq}{\leqslant}
\renewcommand{\ge}{\geqslant}
\renewcommand{\geq}{\geqslant}

\DeclarePairedDelimiter\set{\{}{\}}
\author[Nicolas Bousquet et al.]{Nicolas Bousquet\affiliationmark{1,2}
  \and Wouter Cames van Batenburg\affiliationmark{3}\thanks{Supported by the Belgian National Fund for Scientific Research (FNRS).}\\
  \and Louis Esperet\affiliationmark{4}
  \and Gwena\"el Joret\affiliationmark{3}\thanks{Supported by the Belgian National Fund for Scientific Research (FNRS).}
\and Piotr Micek\affiliationmark{5}\thanks{Supported by the National Science Center of Poland under grant UMO-2023/05/Y/ST6/00079 within the WEAVE-UNISONO program.}}
\title[Shallow brambles]{Shallow brambles}
\affiliation{
LIRIS, CNRS, Universit\'e Claude Bernard Lyon 1, Lyon, France \\
CRM-CNRS, Montreal, Canada \\
Computer Science Department,  Universit\'e libre de Bruxelles, Brussels, Belgium\\
G-SCOP,   CNRS, Universit\'e Grenoble Alpes,  Grenoble, France\\
Faculty of Mathematics and Computer Science,
Jagiellonian University, Kraków, Poland}
\begin{document}
\publicationdata{vol. 27:3}{2025}{12}{10.46298/dmtcs.15257}{2025-02-19; 2025-02-19; 2025-06-18}{2025-09-09}
\maketitle
\begin{abstract}
  A graph class $\mathcal{C}$ has \emph{polynomial expansion} if there is a polynomial function $f$ such that for every graph $G\in \mathcal{C}$, each of the depth-$r$ minors of $G$ has average degree at most $f(r)$. 
In this note, we study bounded-radius variants of some classical graph parameters such as bramble number, linkedness and well-linkedness, and we show that they are pairwise polynomially related. Furthermore, in a  graph class with polynomial expansion they are all uniformly bounded by a polynomial in $r$.
\end{abstract}

\section{Introduction}

Nešetřil and Ossona de Mendez~\cite{sparsity} introduced the notion of graph classes with bounded expansion as a way to capture graphs that are sparse at all scales. 
To measure sparsity, they consider a refinement of the usual notion of graph minors, namely {\em shallow minors}. 
For $r\geq 0$, a graph $H$ is a \emph{depth-$r$ minor} of a graph $G$, denoted $H \preccurlyeq_r G$, if $H$ can be obtained from a subgraph of $G$ by contracting pairwise vertex-disjoint connected subgraphs, each 
of radius at most $r$, into vertices.  
The invariant $\nabla_r(G)$ is then defined as follows: 
\[
\nabla_r(G):= \max \left\{ \frac{|E(H)|}{|V(H)|}: H \preccurlyeq_r G \right\}.
\]
In other words, $\nabla_r(G)$ is the maximum edge density  of a depth-$r$ minor of $G$. 
A graph class $\mathcal{C}$ has \emph{bounded expansion} if there is a function $f$ such that $\nabla_r(G) \leq f(r)$ for every $G\in \mathcal{C}$ and every integer $r\geq 0$. 
Furthermore, $\mathcal{C}$ has \emph{polynomial expansion} if $f$ can be taken to be a polynomial.

The theory of graph classes with bounded expansion is rich and has a number of connections with other areas. 
There are several characterizations of the notion of bounded expansion using different graph parameters. 
Among them, the generalized coloring numbers play a special role as they are usually some of the easiest parameters to work with in practice, see the recent survey of Siebertz \cite{Sie25}.
In this note, we focus on strong coloring numbers, which are defined as follows. 
Given a linear order $\pi$ on the vertices of a graph $G$ and $r\in\mathbb{N}\cup\set{\infty}$, 
we say that a vertex $u$ is \emph{strongly $r$-reachable} from a vertex $v$ if: $\pi(u) \leq \pi(v)$ and there is a path of length at most $r$ joining $v$ and $u$ such that $\pi(v)< \pi(w)$ for every internal vertex $w$ of the path.
The set of vertices that are strongly $r$-reachable from $v$ with respect to $\pi$ is denoted by  $\sreach [r,\pi,v]$. 
The \emph{strong $r$-coloring number} is $$\scol_r(G):= \min_{\pi} \max_{v\in V(G)} |\sreach [r,\pi,v]|,$$
where the minimum is taken over all linear orders $\pi$ of $V(G)$.
Note that
\[\scol_0(G) \leq \scol_1(G) \leq \cdots \leq \scol_{\infty}(G).\]
Moreover, $\scol_1(G)$ is equal to the \emph{coloring number} of $G$ (one plus the degeneracy), whereas on the other side of the spectrum $\scol_{\infty}(G)$ equals the treewidth plus one. 
Zhu~\cite{Zhu09} proved that  the following statements are equivalent for a class of graphs $\mathcal{C}$:
\begin{enumerate}
\item[(i)] there is a function $f$ such that $\nabla_r(G)\leq f(r)$ for all $G$ in $\mathcal{C}$ and all $r\geq0$;
\item[(ii)] there is a function $f$ such that $\scol_r(G)\leq f(r)$ for all $G$ in $\mathcal{C}$ and all $r\geq0$.
\end{enumerate}
This is one of the first characterizations of classes with bounded expansion, among many.

\smallskip

Let us now turn our attention to classes with polynomial expansion. 
Denote by $\omega_r(G)$ the size of a largest clique in a depth-$r$ minor of $G$. 
Then the following statements are equivalent for a class  of graphs $\mathcal{C}$~\cite{SparsityLectureNotes2019}: 
\begin{enumerate}
\item \label{item:nabla} there is a polynomial function $f$ such that $\nabla_r(G)\leq f(r)$ for all $G$ in $\mathcal{C}$ and all $r\geq0$;
\item \label{item:omega} there is a polynomial function $f$ such that $\omega_r(G)\leq f(r)$ for all $G$ in $\mathcal{C}$ and all $r\geq0$.
\end{enumerate}
Moreover, if $\mathcal{C}$ is monotone\footnote{A class of graphs $\mathcal{C}$ is {\em monotone} if any subgraph of a graph from $\mathcal{C}$ is also in $\mathcal{C}$.} then the following is also equivalent to  \eqref{item:nabla} and \eqref{item:omega} (see~\cite{DN16}): 
\begin{enumerate}
\setcounter{enumi}{2}
\item \label{item:separator} there is a real number $\delta>0$ such that every graph $G\in \mathcal{C}$ has a balanced separator of order $O(|V(G)|)^{1-\delta}$.
\end{enumerate}
Having balanced separators of size strongly sublinear in the number of vertices 
is a very useful property in structural graph theory and algorithm design, which motivates the study of graph classes with polynomial expansion. 
However, these classes are far from being well understood. 
Indeed, not much else is known besides the equivalences stated above. 

Given that classes with bounded expansion can be characterized in terms of bounded strong coloring numbers, it is natural to wonder whether the following is equivalent to polynomial expansion:

\begin{enumerate}
\setcounter{enumi}{3}
\item \label{item:scols} there is a polynomial function $f$ such that $\scol_r(G)\leq f(r)$ for all $G$ in $\mathcal{C}$ and all $r\geq0$.
\end{enumerate} 

It is not difficult to show that \eqref{item:scols} implies \eqref{item:nabla}, see e.g.\ \cite{ER2018, Sie25}. 
However, the other direction remains open, which was stated as an open problem by Joret and Wood:\footnote{In~\cite{ER2018} the question is asked for monotone graph classes, however we believe monotonicity is not necessary.}
 
\begin{qn}[see~\cite{ER2018}]\label{conj:weaker} 
Is it true that for every graph class $\mathcal{C}$ with polynomial expansion, there exists a polynomial function $f$ such that for every integer $r\geq 0$ and every graph $G \in \mathcal{C}$,
\[
\scol_r(G) \leq f(r)?
\]
\end{qn}

If true, this would yield a new characterization of polynomial expansion, in terms of polynomial strong coloring numbers.

In this note, we introduce  four new graph parameters that are bounded radius refinements of classical graph parameters: the \emph{\dep-$r$ bramble number}, \emph{\dep-$r$ tangle number}, \emph{\dep-$r$ linkedness} and \emph{\dep-$r$ well-linkedness} of a graph. All these parameters quantify how well-connected a subset of vertices of bounded diameter can be in a given graph. 
Our main result is that, for some absolute constants $c_1, c_2 >0$, these new parameters are all lower bounded by polynomials in $r$ and $\omega_{c_1\cdot r}(G)$, and likewise they are all upper bounded by polynomials in $r$ and $\omega_{c_2\cdot r}(G)$; see Theorem~\ref{thm:summaryresults}. 
In particular, these parameters are all polynomially bounded in $r$ for graphs belonging to a fixed class with polynomial expansion.

It is our hope that these parameters might help in answering  Question~\ref{conj:weaker} positively. 
In particular, to do so it would be enough to show that, for graphs belonging to a fixed class with polynomial expansion, strong $r$-coloring numbers can be bounded from above by a polynomial function of $r$ and the \dep-$r$ bramble number:

\begin{qn}\label{quest:equiv}
Is it true that for every graph class $\mathcal{C}$ with polynomial expansion, there exists a polynomial function $f$ such that for every integer $r\geq 0$ and every graph $G \in \mathcal{C}$, 
\[
\scol_r(G) \leq f(r,\bn_{r}(G))?
\]
\end{qn}

Let us emphasize the following subtlety in Question~\ref{quest:equiv}: The assumption that $G$ belongs to a fixed class with polynomial expansion is essential, because if we consider the set of all graphs, the $r$-strong coloring numbers cannot be bounded from above by {\em any} function of $r$ (polynomial or not) and the \dep-$r$ bramble number. 
This remains true even if we replace the \dep-$r$ bramble number by the \dep-$g(r)$ bramble number for any function $g(r)$, see~\cref{lem:counterexample} in Section~\ref{sec:proof}. 

\medskip

\subsection*{Related work} The topic of this paper bears some superficial similarities with the Coarse Graph Theory project initiated recently by Georgakopoulos and Papasoglu \cite{GP23}, which aims at revisiting structural graph theory (and in particular the theory of graph minors) from the point of view of metric theory. A \emph{fat $K_t$-minor} or \emph{metric $K_t$-minor} in a graph $G$ in this context is a minor-model of $K_t$ in $G$ such that  the branch vertices are far apart, and are connected by paths that are also far apart. Here all distances are considered in the graph $G$. A recent paper by Dragan \cite{Dra25} on related topics  explores the minimum radius $r$ such that for every bramble $\mathcal{B}$ in a graph $G$, there is a ball of radius $r$ in $G$ that intersects all elements of $\mathcal{B}$. One of the results of \cite{Dra25} is that this minimum radius is within a small multiplicative constant of the \emph{tree-length} of $G$, the minimum over all tree-decompositions of $G$ of the maximum diameter of a bag.

This is somewhat orthogonal to the approach we take here, where minor-models and bramble elements have small radius (where distances are computed in the graph induced by the branch vertices or elements, not in the original graph). Nevertheless, this paper can be seen as an effort to understand what remains true if we introduce appropriate distance constraints in classical objects from graph minor theory such as minor-models, brambles, tangles, and well-linked sets.

\medskip

It was pointed out to us by Micha\l\ Pilipczuk that our definition of a shallow bramble is related to the \emph{cop-width} of a graph, a parameter recently introduced by Toru\'nczyk \cite{Tor23}. Note that very much like brambles (or heavens) in the helicopter version of the cops and robber game of Seymour and Thomas \cite{ST1993}, the existence of a depth-$r$ bramble of order $k$ provides  a winning strategy for the robber against $k-1$ cops in the version of the game by Toru\'nczyk  when the robber has speed $2r+1$. 

\subsection*{Organization of the paper} We start with the definitions of all parameters in Section \ref{sec:def}. In Section \ref{sec:res} we state the main existing results on the connections between these parameters in the classical $r=\infty$ setting, and then state our main result for the general case, that is, for an arbitrary $r\geq 0$. 
Our main result, Theorem \ref{thm:summaryresults}, is then proved in Section \ref{sec:proof}. We conclude in Section \ref{sec:ccl} with a discussion on other natural bounded radius refinements of classical parameters that might be relevant to study in this context.

\section{Definitions and notation}\label{sec:def}

All graphs in this paper are finite, simple, and undirected. 
For a vertex $v$ in a graph $G$ and an integer $r\geq 0$, the \emph{$r$-ball in $G$ centered at $v$} is the set of all vertices of $G$ that are at distance at most $r$ from $v$. 

We say that a subset $B\subseteq V(G)$ is connected if the induced subgraph $G[B]$ is.
A {\em minor-model} of a graph $H$ in a graph $G$ is a collection $\mathcal{M}=\{M_v: v\in V(H)\}$ of pairwise disjoint connected subsets of $V(G)$, one for each vertex $v$ of $H$, such that there is an edge in $G$ between $M_v$ and $M_w$ for every edge $vw\in E(H)$. 
If moreover $G[M_v]$ has radius at most $r$ for every $v\in V(H)$, we say that $\mathcal{M}$ is  a {\em depth-$r$} minor-model of $H$ in $G$.

Two vertex subsets of a graph $G$ are said to \emph{touch} if they have a vertex in common or are joined by an edge of $G$.  A \emph{bramble} $\mathcal{B}$ in $G$ is a set of pairwise touching connected subsets of $V(G)$, called elements of $\mathcal{B}$. Note that a bramble can be viewed as a generalization of a minor-model of a complete graph, with the difference that its elements are allowed to have vertices in common.  A \emph{hitting set} of $\mathcal{B}$ is a set of vertices $X$ such that $B \cap X \neq \emptyset$ for each $B\in \mathcal{B}$. The \emph{order} of $\mathcal{B}$ is the minimum size of such a hitting set. Finally, the \emph{bramble number of $G$}, denoted $\bn_\infty(G)$, is the maximum order of a bramble in $G$. A classical result of Seymour and Thomas \cite{ST1993} states that $\bn_\infty(G)$ is equal to the treewidth of $G$ plus one.

\smallskip

We consider two successive refinements of brambles:

\begin{itemize}
\item  For $r\in \mathbb{N}\cup \{\infty\}$, a \emph{\dep-$r$ bramble} of a graph $G$ is a bramble such that each of its bramble elements induces a subgraph of $G$ of radius at most $r$.
The maximum order of a \dep-$r$ bramble of $G$ is the \emph{\dep-$r$ bramble number of $G$}, which we denote by $\bn_r(G)$. 

\item For $t\geq 1$ and $r \in \mathbb{N}\cup \{\infty\}$, a \emph{\dep-$r$ $t$-bramble} is a \dep-$r$ bramble $\mathcal{B}$ such that the intersection of any $t$ (not necessarily distinct) elements of $\mathcal{B}$ is non-empty.
The maximum order of a \dep-$r$ $t$-bramble of a graph $G$ is the \emph{\dep-$r$ $t$-bramble number of $G$}, denoted by $\bn_{r,t}(G)$.
\end{itemize} 

The notion of \dep-$r$ $t$-bramble is of interest because for $t=3$ and $r=\infty$, it is closely related to a \emph{tangle}. Tangles are usually defined in terms of separations, but as demonstrated in~\cite{Reed}, the following yields an equivalent definition.
For $r\in \mathbb{N}\cup \{\infty\}$, a \emph{\dep-$r$ tangle} $\mathcal{T}$ of a graph $G$ is a \dep-$r$ bramble of $G$ such that for any three (not necessarily distinct) elements $T_1,T_2,T_3$ of $\mathcal{T}$, either
\begin{enumerate}[label={(\roman*)}]
\item $T_1\cap T_2 \cap T_3$ is non-empty, or
\item there is an edge $e\in E(G)$ such that all of $T_1, T_2, T_3$ contain an endpoint of $e$. 
\end{enumerate}
The \emph{order} of a \dep-$r$ tangle $\mathcal{T}$ is the order of $\mathcal{T}$ viewed as a bramble.  
The \emph{\dep-$r$ tangle number $\tangle_r(G)$} of a graph $G$ is the maximum order of a \dep-$r$ tangle in $G$.
The \emph{tangle number} of $G$ is $\tangle_{\infty}(G)$.

Observe that every \dep-$r$ $3$-bramble is a \dep-$r$ tangle, and every \dep-$r$ tangle is a \dep-$r$ bramble, so we have $\bn_{r,3}(G) \leq \tangle_r(G) \leq \bn_r(G)$ for every graph $G$.

Let $r\in\mathbb{N}\cup\set{\infty}$ and let $k$ be an integer with $k\geq 1$. 
We say that a vertex set $S$ in a graph $G$ is \emph{\dep-$r$ $k$-linked} if for any set $X$ of fewer than $k$ vertices, there is a connected subset in $G-X$ of radius at most $r$ that contains more than half of the vertices of $S$. 
The \emph{\dep-$r$ linkedness} of $G$, denoted $\link_r(G)$, is the largest integer $k$ for which $G$ contains a \dep-$r$ $k$-linked set. The \emph{linkedness} of $G$ is $\link_{\infty}(G)$.

We say that a set $S$ of vertices is \emph{\dep-$r$ well-linked} if for every two subsets $A$ and $B$ of $S$ with $|A|=|B|$ and for every vertex set $Y$ of size less than $|A|$, there exists an $A-B$ path of length $\leq r$ in $G-Y$.
The \emph{\dep-$r$ well-linkedness} of $G$, denoted $\WL_r(G)$, is the size of the largest \dep-$r$ well-linked set in $G$.
The \emph{well-linkedness} of $G$ is $\WL_{\infty}(G)$.

We note that in the case $r=\infty$, by Menger's theorem, the notion of a \dep-$r$ well-linked set could be equivalently defined by requiring the existence of $|A|$ vertex-disjoint $A-B$ paths (of length at most $r$) for every two equal-sized subsets $A$ and $B$ of $S$. However, for finite $r$ this equivalence is far from correct, due to examples in~\cite{LNP1978}.

\section{Main results}\label{sec:res}

The following are classical results about the $r=\infty$ case of the parameters under consideration in this note. 

\begin{thm}[\cite{Reed, RobertsonSeymourX}]\label{thm:summaryliterature}
For every graph $G$,
\begin{enumerate}[label={\rm{(\roman*)}},itemsep=0pt, parsep=0pt, before=\setlength{\baselineskip}{8mm}]
\item $\scol_{\infty}(G)=\bn_{\infty}(G)$;
\item $\omega_{\infty}(G) \leq \bn_{\infty}(G)$;
\item $\link_{\infty}(G)\leq \bn(G)\leq 2\cdot \link_{\infty}(G)$;
\item $\bn_{\infty}(G)\leq \WL_{\infty}(G)\leq 4\cdot \bn_{\infty}(G)$;
\item $\tangle_{\infty}(G) \leq \bn_{\infty}(G)\leq \frac{3}{2}\cdot \tangle_{\infty}(G)$.
\end{enumerate}
\end{thm}

\medskip
Our main result is the following theorem, which can be seen as a bounded radius refinement of Theorem~\ref{thm:summaryliterature}. 

\begin{thm}\label{thm:summaryresults}
For every graph $G$ and all integers $r\geq 0$ and $t\geq 1$,
\begin{enumerate}[label={\rm{(\roman*)}},itemsep=0pt, parsep=0pt, before=\setlength{\baselineskip}{8mm}]
\item \label{item:one} $\omega_r(G)\leq \bn_r(G)\leq \scol_{4r+1}(G)$;
\item $ \bn_r(G) \leq (5r+1)\cdot (\omega_{5r+1}(G))^2$;
\item $\link_r(G) \leq \bn_r(G)$;
\item $\bn_r(G) \leq 2 \cdot \link_{3r+1}(G)$;
\item $\bn_r(G) \leq \WL_{4r+1}(G)$;
\item $\WL_r(G) \leq 4 \cdot (1+\link_{3r}(G))^2$;
\item $\bn_{r,t}(G) \leq \bn_r(G) \leq t \cdot \bn_{3r+1,t}(G)$;
\item $\bn_{r,3}(G) \leq \tangle_r(G) \leq \bn_r(G)$.
\end{enumerate}
\end{thm}

If we are happy with slightly less optimized bounds, we can summarize most of Theorem~\ref{thm:summaryresults} in a single sequence of inequalities, for instance as follows (for $r\geq 1$):

$$\omega_r \leq \WL_{5r} \leq 16 \cdot \link_{15r}^2 \leq 50 \cdot \tangle_{46r}^2 \leq 50 \cdot \bn_{46r}^2 \leq 10^7 r^2 \cdot \omega_{250 r}^4 .$$

In particular, the parameters $\omega_r, \bn_r, \WL_r, \link_r, \tangle_r, \bn_{r,t}$ are pairwise polynomially upper and lower bounded in terms of each other and $r$ (and $t$).

\section{The proofs}\label{sec:proof}

We start with the following lemma mentioned in the introduction, showing that if we consider the set of all graphs,
the strong $r$-coloring numbers cannot be bounded from above by any function of $r$ and the \dep-$g(r)$ bramble number, for any function $g(r)$. 

\begin{lem}\label{lem:counterexample}
For all integers $d,r,s \ge 1$, there exists a graph $G$ such that $\scol_r(G) \geq d$ and $\bn_s(G)=2$.
\end{lem}
\begin{proof}
Due to a classical result of Erd\H{o}s  \cite{Erd59}, there exists a graph $G$ with minimum degree at least $d$ and girth at least $8s+4$. Then $\scol_r(G) \geq \scol_1(G) \geq d+1$. 
It remains to show that $\bn_{s}(G)=2$. 

Let $\mathcal{B}$ be a \dep-$s$ bramble in $G$. 
Since any two bramble elements touch and each bramble element has radius at most $s$,
it follows that the graph $G(\mathcal{B})$ induced by the union of the elements of $\mathcal{B}$ has diameter at most $4s+1$. Since this is less than half the girth of $G$, $G(\mathcal{B})$ must induce a tree. Trees have treewidth $1$, so every bramble (in particular every \dep-$s$ bramble) has a hitting set of size at most $2$~\cite{ST1993}. On the other hand, any edge $uv$ gives rise to the \dep-$s$ bramble $\mathcal{B}:=\{u,v\}$ that has order $2$, so $\bn_{s}(G)=2$. 
\end{proof}

We now turn to the proof of Theorem~\ref{thm:summaryresults}. We remark that several of the arguments draw inspiration from those in~\cite{Reed} for the $r=\infty$ case.
We start with item~\ref{item:one} of Theorem~\ref{thm:summaryresults}, and then proceed in order.
\begin{thm}
For every graph $G$ and integer $r\geq0$, 
$$\omega_r(G)\leq \bn_r(G)\leq \scol_{4r+1}(G).$$
\end{thm}
\begin{proof}
The first inequality holds because every depth-$r$ minor-model of a complete graph $K_n$ is also a \dep-$r$ bramble of order $n$.

For the second inequality, let $\mathcal{B}$ be a \dep-$r$ bramble of $G$ and let $\pi$ be an ordering of $V(G)$. 
For each bramble element $B\in \mathcal{B}$, let $v_B$ be the minimum vertex of $B$ w.r.t.\ $\pi$, and let $v^*$ be the maximum vertex among all $v_B$  w.r.t.\ $\pi$. More formally, $v_B$ is the vertex such that $\pi(v_B)=\min_{w\in B} \pi(w)$ and $v^*$ is such that $\pi(v^*)= \max_{B \in \mathcal{B}} \min_{w\in B} \pi(w)$. 
Since $\mathcal{B}$ is a \dep-$r$ bramble, each bramble element of $\mathcal{B}$ has diameter at most $2r$. Thus for every $B\in \mathcal{B}$ there is a path $P(v^*,B)$ from $v^*$ to $v_B$ of length at most $4r+1$. Traversing $P(v^*,B)$ starting from $v^*$, let $x_B$ be the first vertex such that $\pi(x_B)\leq \pi(v^*)$. (Note that $x_B$ indeed exists since $\pi(v_B)\leq \pi(v^*)$.)
Then $X:=\bigcup_{B\in \mathcal{B}} x_{B}$ is both a hitting set of $\mathcal{B}$ and a subset of $\sreach[4r+1,\pi,v^*]$. Since we derived this for arbitrary $\pi$ and $\mathcal{B}$, it follows that $\bn_r(G) \leq \scol_{4r+1}(G)$.
\end{proof}

\begin{thm}
For every graph $G$ and integer $r\geq0$, 
 $$  \bn_r(G) \leq (5r+1)\cdot (\omega_{5r+1}(G))^2.$$
\end{thm}
\begin{proof}
 We prove that $\bn_r(G) \leq (5r+1) \cdot t^2$, where $t:=\omega_{5r+1}(G)$.
Let $\mathcal{B}$ be a \dep-$r$ bramble of $G$.
We call $\mathcal{M}$ a \emph{good model} if
\begin{itemize}
\item $\mathcal{M}$ is a depth-$(5r+1)$ minor-model of a complete graph in $G$, and
\item $|V(M)|\leq 1+5(r+1)(t-1)$ for every branch set $M\in \mathcal{M}$.
\end{itemize}

Note that there exists at least one good model, since a single vertex defines one. Given a minor-model $\mathcal{M}$, let $V(\mathcal{M})$ denote the union of the branch sets of $\mathcal{M}$, and let $R(\mathcal{M})$ denote the union of all bramble elements of $\mathcal{B}$ that are disjoint from $V(\mathcal{M})$. Observe that, since every two bramble elements touch, the subgraph induced by $R(\mathcal{M})$ has radius at most $3r+1$.

\smallskip

From now on, we fix a good model $\mathcal {M}$ in $G$ such that
$(|R(\mathcal{M})|, |V(\mathcal{M})|)$ is lexicographically minimized among all good models. Let $M_1,M_2,\ldots, M_s$ be the branch sets of $\mathcal{M}$. Note that by the definition of $t$, we have $s\leq t$.

\smallskip

The proof is done if $R(\mathcal{M})= \emptyset$, because then $V(\mathcal{M})$ is a hitting set of $\mathcal{B}$ of order at most \[\sum_{1=1}^{s} |V(M_i)| \leq t+ (5r+1)(t-1)t \leq (5r+1)\cdot t^2.\]
Thus we may assume that $R(\mathcal{M})$ is non-empty, and from there we will derive a contradiction.
\smallskip

For every branch set $M_i\in \mathcal{M}$ it holds that
 \begin{equation}\label{ineq:}
 |R(\mathcal{M}-M_i)| > |R(\mathcal{M})|, 
 \end{equation}
  for otherwise we would have $|R(\mathcal{M}-M_i)| \le|R(\mathcal{M})|$ and $|V(\mathcal{M}-M_i)| <|V(\mathcal{M})|$, contradicting the minimality of $\mathcal{M}$.

Thanks to (\ref{ineq:}) we know that each branch set $M_i\in \mathcal{M}$ satisfies the following key
property: There is a bramble element $B\in\mathcal{B}$ that intersects $M_i$ and avoids all other branch sets of $\mathcal{M}$. Take a shortest path $P_i$ from $M_i$ to the non-empty set $R(\mathcal{M})$ through $M_i \cup B$. Let $x_i$ be the endpoint of $P_i$ in $R(\mathcal{M})$. Observe that $P_i$ contains at most $2r+1$ edges, and thus at most $2r+2$ vertices. Let $Q_i$ be $P_i$ minus its endpoint in $M_i$ (so $Q_i$ contains at most $2r$ edges).

Recall that the subgraph induced by $R(\mathcal{M})$ has radius at most $3r+1$, so there exists a vertex $x^*\in R(\mathcal{M})$ and for each $1\le i \le s$, a path on at most $3r+1$ edges, included in $R(\mathcal{M})$ and connecting $x^*$ to  $x_i$. Let $X\subseteq R(\mathcal{M})$ be the union of the vertex sets of these $s$ paths. Note that $X$ is connected and contains at most $1+(3r+1)s$ vertices. 

Now define a new branch set $M_{s+1}$, consisting of $X$ together with the paths $Q_1,\ldots, Q_s$. Then the radius of $M_{s+1}$ is at most $5r+1$, and $|V(M_{s+1})|\leq 1+(5r+1)s\le 1+(5r+1)(t-1)$. 
Moreover, by construction $M_{s+1}$ is connected to all previous branch sets. We conclude that $\mathcal{M}+M_{s+1}$ is a good model with $|R(\mathcal{M}+M_{s+1})|<|R(\mathcal{M})|$, contradicting the minimality of $\mathcal{M}$.
\end{proof}

\begin{thm}
For every graph $G$ and integer $r\geq 0$,
$$\link_r(G) \leq \bn_r(G)$$
\end{thm}
\begin{proof}
Let $S$ be a \dep-$r$ $k$-linked set, where $k:=\link_r(G)$.
By definition of \dep-$r$ $k$-linkedness, for each set $X\subseteq V(G)$ of order $|X|<k$, there exists an $r$-ball $B_X$ in $G-X$ that contains more than half of the vertices of $S$. Let $\mathcal{B}:=(B_X)_{X \subseteq V(G), |X|<k}$ be the set of these balls. The elements of $\mathcal{B}$ pairwise intersect since each of them contains more than half of $S$, and therefore $\mathcal{B}$ is a \dep-$r$ bramble of $G$. 
Furthermore, the order of $\mathcal{B}$ is at least $k$, since any hypothetical hitting set $X$ of order less than $k$ would not hit the bramble element $B_{X}$. Thus $\bn_r(G) \geq k$, as desired.
\end{proof}

\begin{thm}For every graph $G$ and integer $r\geq 0$,
$$\bn_r(G) \leq 2 \cdot \link_{3r+1}(G).$$
\end{thm}
\begin{proof}
Let $S$ be a minimum order hitting set for a \dep-$r$ bramble $\mathcal{B}$. It suffices to show that $S$ is \dep-$(3r+1)$ $(\lceil |S|/2 \rceil)-$linked. To that end, let $Y$ be a set of fewer than $\lceil |S|/2 \rceil$ vertices. 

Since $Y$, being smaller than $S$, is not a hitting set for $\mathcal{B}$, there exists a bramble element $B \in \mathcal{B}$ which is disjoint from $Y$. Because $B$ is an element of a \dep-$r$ bramble, we can choose a vertex $u$ in $B$ such that (in the graph induced by $B$ and thus also in $G-Y$) every vertex of $B$ is at distance at most $r$ from $u$. We let $U$ denote the $(3r+1)$-ball in $G-Y$ centered at $u$. It suffices to show that $U$ contains more than half of $S$, which we will now do.

Suppose for a contradiction that $|U\cap S|\leq \lfloor \frac{|S|}{2} \rfloor$.
If there exists any, let $C\in \mathcal{B}$ be a bramble element that is disjoint from $Y$. As $S$ hits $C$, we may choose a vertex $c \in C \cap S$. Since $C$ and $B$ touch and $C$ has radius at most $r$, there is a path in $G-Y$ joining $u$ and $c$ of length at most $3r+1$, so that $c\in U \cap S$. In other words, $U\cap S$ hits $C$.
 
It follows that $(U\cap S)\cup Y$ is a hitting set for $\mathcal{B}$ of size
$|(U\cap S)|+|Y| < \lfloor \frac{|S|}{2} \rfloor + \lceil \frac{|S|}{2} \rceil = |S|$, contradicting the minimality of $S$.
\end{proof}

\begin{thm}
For every graph $G$ and integer $r\geq 0$,
$$\bn_r(G) \leq \WL_{4r+1}(G).$$
\end{thm}
\begin{proof}
We prove that any minimum hitting set $S$ for a \dep-$r$ bramble $\mathcal{B}$ is \dep-$(4r+1)$ well-linked.
Suppose not, then there exist two equal-sized subsets $A, B$ of $S$ and a set $Y$ of fewer than $|A|$ vertices, such that $G-Y$ does not contain any $A-B$ path of length $\leq 4r+1$.
Consider the sub-bramble $\mathcal{B}_Y$ of $\mathcal{B}$ that contains all bramble elements of $\mathcal{B}$ that are disjoint from $Y$. If both $A$ and $B$ intersect some (not necessarily the same) bramble element of $\mathcal{B}_Y$ then these bramble elements touch, so then there must be a path of length at most $4r+1$ between $A$ and $B$ in $G-Y$, a contradiction. So without loss of generality $\bigcup_{C \in \mathcal{B}_Y} C \cap A$ is empty. But then $(S-A)+Y$ is a hitting set for $\mathcal{B}$, contradicting the minimality of $S$.
\end{proof}

\begin{thm}
For every graph $G$ and integer $r\geq 0$,
$$\WL_r(G) \leq 4 \cdot (1+\link_{3r}(G))^2.$$
\end{thm}
\begin{proof}
It suffices to show that any \dep-$r$ well-linked set $S$ is \dep-$3r$ $k_S$-linked, where $k_S:=\lfloor \tfrac{1}{2}\sqrt{|S|} \rfloor$.

Assume for a contradiction that there exists a set $X$ of fewer than $k_S$ vertices such that every $3r$-ball in $G-X$ contains at most half of the vertices of $S$.

\smallskip
We first consider the case that $G-X$ contains a $2r$-ball $B_{2r}(v)$, centered at some vertex $v$, containing more than $k_S$ vertices of $S$. Let $B_{3r}(v)$ be the $3r$-ball in $G-X$ centered at the same vertex. Since $|B_{3r}(v)\cap S|\leq \lfloor |S|/2 \rfloor $, we know that $S-(X \cup B_{3r}(v))$ has size at least  $|S| - (k_S-1)- \lfloor |S|/2 \rfloor \geq k_S$.
Thus we can choose a subset $A$ of $B_{2r}(v)\cap S$ and a subset $B$ of $S-(X \cup B_{3r}(v))$ such that $|A|=|B|=k_S>|X|$. However, by construction there cannot exist any $A-B$ path of length at most $r$ in $G-X$, contradicting our assumption that $S$ is \dep-$r$ well-linked.

\smallskip
 We have thus reduced to the case that no $2r$-ball in $G-X$ contains more than $k_S$ vertices of $S$.

 \smallskip
 
We now greedily build a collection of $r$-balls in $G-X$, collecting at least one new vertex from $S$ at each step, and stopping as soon as the corresponding $2r$-balls (again in $G-X$ and with the same centers) cover at least half of $S$. Let $U_{r}$ denote the union of the $r$-balls, and let $U_{2r}$ denote the union of the corresponding $2r$-balls.

Since each of the $2r$-balls contains at most $k_S$ vertices from $S$, we need at least $(|S|/2) / k_S \geq k_S$ balls. Combining this with the fact that each $r$-ball adds at least one new vertex of $S$, we conclude that 
$$|U_{r} \cap S| \geq k_S.$$

On the other hand, note that 
$|U_{2r} \cap S| < \frac{|S|}{2} + k_S,$ since we stopped adding balls as soon as the union was at least $\frac{|S|}{2}$.
Therefore
$$|S- (X \cup U_{2r})| \geq |S|-(k_S-1)- \left( \frac{|S|-1}{2} + k_S  \right) \geq k_S.$$

Thus we can choose a subset $A$ of $U_{r} \cap S$ and a subset $B$ of $S- (X \cup U_{2r})$ such that $|A|=|B|=k_S> |X|$. However, as in the first case, every $A-B$ path in $G-X$ must have length more than~$r$, contradicting that $S$ is \dep-$r$ well-linked.
\end{proof}

\begin{thm}
For every graph $G$ and integers $r\geq 0, t\geq 1$,
$$\bn_{r,t}(G) \leq \bn_r(G) \leq t\cdot \bn_{3r+1,t}(G).$$
\end{thm}
\begin{proof}
The lower bound on $\bn_r(G)$ is true because every \dep-$r$ $t$-bramble is a \dep-$r$ bramble, so we only need to prove the upper bound.
Let $\mathcal{B}$ be a \dep-$r$ bramble of order $k:=\bn_r(G)$.
For each $X\subseteq V(G)$ with $|X|< \lceil \tfrac{k}{t} \rceil$, we can choose a bramble element $B_X$ of $\mathcal{B}$ that is disjoint from $X$, since $X$ is too small to be a hitting set of $\mathcal{B}$. Let $v_X$ be a vertex of $B_X$ such that $B_X$ is contained in the $r$-ball centered at $v_X$. Now define $T_X$ as the $(3r+1)$-ball in $G-X$ centered at $v_X$. 
We claim that $\mathcal{T}:= \left( T_X \right)_{X \subseteq V(G), |X| < \tfrac{k}{t} }$ is a \dep-$(3r+1)$ $t$-bramble of $G$ of order at least $\lceil \tfrac{k}{t} \rceil$. (This then implies that $k\leq t\cdot \bn_{3r+1,t}(G)$, as desired.)

Indeed, consider $t$ sets $X_1, \ldots, X_t$ of size smaller than $\lceil \tfrac{k}{t} \rceil$. 
Since $|\bigcup_{1\leq i \leq t} X_i|< k$, there is an element $B$ of $\mathcal{B}$ that is disjoint from $\bigcup_{1\leq i \leq t} X_i$.  For each $1\leq i \leq t$, because $B_{X_i}$ and $B$ touch, are disjoint from $X_i$ and both have radius at most $r$, we see that every vertex of $B$ can be reached from $v_{X_i}$ by some path in $G-X_i$ of length at most $3r+1$, so that $B \subseteq T_{X_i}$.
Thus $B \subseteq \bigcap_{1\leq i \leq t} T(X_i)$. In particular, the intersection of $T(X_1),\ldots, T(X_t)$ is non-empty, so $\mathcal{T}$ is indeed a \dep-$(3r+1)$ $t$-bramble. Finally,  $\mathcal{T}$ does not have any hitting set $Y$ of size strictly smaller than $\lceil \tfrac{k}{t} \rceil$, because $Y$ would not hit the element $T_{Y}$. So $\mathcal{T}$ has order at least $\lceil \tfrac{k}{t} \rceil$.
\end{proof}

\section{Final remarks}\label{sec:ccl}

Since large treewidth (and hence large $\scol_{\infty}$) is equivalent to the existence of a large grid minor~\cite{RobertsonSeymourV}, one may wonder what is the relation between $\scol_r$ and depth-$r$ grid minors for finite~$r$, and whether it helps to answer the questions in this paper. For a graph $G$, let $\GR_r(G)$ denote the largest integer $t$ such that $G$ contains the $t \times t$ grid as a depth-$r$ minor. \\

Any $t \times t$ grid is a (depth-$0$) minor of $K_{t^2}$, so we have $\omega_r \leq (\GR_{r})^2$. However, $\GR_r$ cannot be upper bounded by a function of $\omega_r$ and $r$, as can be seen by considering large grids, which have no $K_5$-minor. Thus the behavior of $\GR_r$ is really different from that of $\omega_r$, $\bn_r$, $\link_r$ and $\WL_r$, even when restricted to classes of polynomial expansion.\\

So could grid minors instead be closely related to strong colorings?
The answer is no again, in the sense that it is not possible to upper bound $\scol_r$ in terms of $\GR_r$ either.
The proof is very similar to that of Lemma~\ref{lem:counterexample}.
In fact, the same type of argument shows that for every $r, k\geq 1$ and for \emph{every} graph $H$ which is not a forest, there are graphs without $H$ as a depth-$k$ minor but with arbitrarily large strong $r$-coloring number.\\

Thus we see that for finite $r$, the grid minors are not as intimately related to  $\scol_r$ as they are for infinite $r$. 
A question that remains is whether this connection does survive when restricted to classes of polynomial expansion.

\begin{qn}
Is it true that for every graph class $\mathcal{C}$ with polynomial expansion, there exists a polynomial function $f$ such that for every integer $r\geq 0$ and every graph $G \in \mathcal{C}$, 
\[
\scol_r(G)\leq f(r,\GR_{\poly(r)}(G))?
\]
\end{qn}

We have seen that there is a duality between the bramble number and tree-width~\cite{ST1993}. A natural question is whether a similar duality holds between the bounded radius variants of these parameters (in the tree-width case, we could ask for a tree-decomposition in which each bag has bounded radius). As above, large grids show that such duality cannot exist: the depth-$r$ bramble number is bounded by a function of $\omega_r$, which is at most 4 for any planar graph (and in particular for grids). On the other hand, grids have arbitrarily large tree-width, regardless of any condition on the radius of the bags.

\acknowledgements
This work was initiated during the online workshop  \textit{Generalized coloring
numbers and friends}, which took place in February 2021 and was co-organized by Micha\l{} Pilipczuk and the last author, following the \emph{Sparse Graphs Coalition} initiative of Ross Kang and Jean-S\'ebastien Sereni. 
We are grateful to Raphael Steiner for interesting discussions that led to \cref{lem:counterexample}. We also thank the reviewers for their helpful comments and suggestions.

\nocite{*}
\bibliographystyle{abbrvnat}

%

\end{document}